\documentclass[11pt]{article}

\usepackage{amsmath,amssymb,amsthm}
\usepackage{graphicx}
\usepackage{subfigure}
\usepackage{enumerate}
\usepackage{fullpage}
\usepackage{url}

\newtheorem{theorem}{Theorem}

\newtheorem{conjecture}{Conjecture}

\newtheorem{lemma}[theorem]{Lemma}

\newcommand{\answerCommand}{}%
  {\renewcommand{\answerCommand}{#1\\}%
   \noindent\textbf{\answerCommand}%
   }{\\}%

  {\renewcommand{\answerCommand}{#1}%
   \noindent\textbf{\answerCommand}%
   }{\\}%

\title{Drawing Graphs with Orthogonal Crossings}
\author{Radoslav Fulek\thanks{Ecole Polytechnique F\'ed\'erale de Lausanne. Email:~\texttt{radoslav.fulek@epfl.ch}}
 \and Bal\'{a}zs Keszegh\thanks{Alfr\'{e}d R\'{e}nyi Institute of Mathematics, Ecole Polytechnique F\'ed\'erale de Lausanne.
 Partially supported by grant OTKA NK 78439. Email:~\texttt{keszegh@renyi.hu}} \and
Filip Mori\'{c} \thanks{Ecole Polytechnique F\'ed\'erale de Lausanne. Email:~\texttt{filip.moric@epfl.ch}}
}
\date{}

\begin{document}

\maketitle

\thispagestyle{empty}

\begin{abstract}
By a poly-line drawing of a graph $G$ on $n$ vertices we understand a drawing of $G$ in the plane such 
that each edge is represented by a polygonal arc joining its two respective vertices.
We call a turning point of a polygonal arc the bend.
We consider the class of graphs that admit a poly-line drawing, in which each edge has 
at most one bend (resp. two bends) and any two edges can cross only at a right angle. It is shown that the number of edges
of such graphs is at most $O(n)$ (resp. $O(n\log^2 n)$).
This is a strengthening of a recent result of Didimo et al.
\end{abstract}



\section{Introduction}

A lot of research in geometric graph theory has been inspired by the problem of making good and easily readable drawings of graphs. Recent cognitive experiments showed that poly-line graph drawings with orthogonal crossings and a small number of bends per edge are equally well readable as planar drawings. Motivated by these findings, Didimo et al. in \cite{Didimo} initiated the study of the classes of graphs which admit such a drawing.

More formally, by a (simple) graph $G=(V,E)$ we understand a pair consisting of the finite
set of vertices $V$ and the finite set of edges $E$ such that $E\subseteq {V \choose 2}$.
By a multigraph $G$ we understand the graph, whose set of edges can be a multiset.

We call a turning point of a polygonal arc the \emph{bend}.
A graph $G$ belongs to the class $R_i$, $i\in \mathbb{N}_0$, if it can be represented in the plane
such that the vertices in $V$ are drawn as points, and the edges in $E$ are drawn as polygonal arcs
with at most $i$ bends joining the vertices, so that any two polygonal arcs representing edges
cross at a right angle (and not at a bend). Obviously, $R_0\subseteq R_1 \subseteq R_2 \subseteq R_3$.

In \cite{Didimo} it is shown that all graphs are in $R_3$, thus $R_i$ for $i\ge 3$ equals to $R_3$, and also they prove that $R_2\subsetneq R_3$. Moreover,
they proved that a graph on $n$ vertices belonging to the class $R_0$, $R_1$, and $R_2$, respectively,
can have at most $O(n)$, $O(n^{4/3})$, and $O(n^{7/4})$ edges. 

We significantly strengthen the above results, and thereby we show that no graphs belonging to $R_2$ have much more than linearly many edges.

\begin{theorem}
\label{thm:R1}
A graph $G$ on $n$ vertices belonging to the class $R_1$ can have at most $O(n)$ edges.
Moreover, there are infinitely many graphs in $R_1$ which do not belong to $R_0$.
\end{theorem}

\begin{theorem}
\label{thm:R2}
A graph $G$ on $n$ vertices belonging to the class $R_2$ can have at most $O(n\log^2n)$ edges.
\end{theorem}

Moreover, we conjecture that the gap between $R_2$ and $R_3$ is even bigger:

\begin{conjecture}
\label{conj:R2}
A graph  $G$  on $n$ vertices belonging to the class $R_2$ can have at most $O(n)$ edges.
\end{conjecture}

A recent paper of Dujmovi\'{c} et al. \cite{Dujmovic} treats a similar question. They proved
an upper bound  on the number of edges in a geometric graph (i.e. its edges are represented by
straight-line segments), in which every pair of edges cross at an angle at least $0<\alpha\leq \pi/2$ for some fixed $\alpha$.

Throughout the paper let $G=(V,E)$ denote a simple graph on $n$ vertices, having $m$ edges.
If $v\in V$, we let $d_v$ denote the number of edges $e$ in $E$ incident to $v$, i.e. the number 
of edges $e$ such that $v\in e$, or shortly the degree of $v$.
By a drawing of a graph $G$ in the plane we understand a representation of the graph in the plane such that
the vertices of $G$ are represented by points, and each edge is represented by a Jordan arc
connecting two points corresponding to its two vertices.

 If it leads to no confusion,
we will refer to the vertices and edges in $G$ also as to the objects that represent them in the drawing.

By a  {\it crossing} of the two edges $e$ and $e'$ in a drawing  of $G$ we understand a point distinct
from the endpoints of $e$ and $e'$ in the intersection $e\cap e'$. By a {\it plane graph} we understand a graph drawn in the plane
without any (edge) crossings.
The \emph{crossing number} $cr(G)$ of a graph $G$ is defined as the minimum number of crossings in a drawing of $G$ over all possible drawings of $G$ in the plane. 
We denote by $F$ the set of faces
of the plane graph $G$. If $f\in F$, we let $d_f$ denote the number of edges on the boundary of $f$, i.e. the size of $f$.
We call a face of size two a {\it lens} (in case of multigraphs), of size three a {\it triangle}, and of size four a {\it quadrangle}.

By a {\it rotation system} at a vertex $v$ of $G$ in a planar representation of $G$ we understand the circular order in which the edges leave
$v$. By a {\it wedge} at a vertex $v$ of $G$ in a planar representation of $G$ we understand a pair of edges $(e,e')$ incident to $v$ that
are consecutive in its rotation system. A face $f$  of a plane graph $G$ contains a wedge $(e,e')$, $e,e'\in E$, if $f$ contains $v$, $e$ and $e'$
on the boundary. Note that a wedge is contained only in one face except when $v$ has degree $2$.

The \emph{bisection width} of $G$ we define as $$b(G)=\min_{|V_1|,|V_2|\leq2n/3}|E(V_1,V_2)|,$$
where the minimum is taken over all partitions $V(G)=V_1\cup V_2$ such that $|V_1|,|V_2|\leq2n/3$ and 
$E(V_1,V_2)$ denotes the set of edges with one endpoint in $V_1$ and the other endpoint in $V_2$.

\section{Discharging}

The method of discharging was apparently introduced in \cite{Wernicke}.
However, it grabbed a considerable attention only after 
it was extensively used in the first valid proof of the famous Four Color Theorem \cite{Appel2}.
Since then it was successfully applied to obtain various types of results
in the structural graph theory, see e.g. \cite{Hlineny}.
Our application of this method reminds that of \cite{Ackerman}.

In order to illustrate our approach we give a short proof of a result that is only slightly weaker
than a result in \cite{Didimo},  which states that a graph on $n$ vertices in $R_0$ can have at most $4n-10$ edges.
Almost the same proof was given quite recently in \cite{Dujmovic}.

\begin{theorem}
\label{thm:R0}
A graph $G$ on $n$ vertices belonging to the class $R_0$ can have at most $4n-8$ edges.
\end{theorem}

\begin{proof}
Fix a drawing $D$ of $G$ in the plane witnessing its membership in $R_0$. We denote by $G'$ the plane
graph which is naturally obtained from $D$ by introducing a new vertex instead of each edge crossing.
Thus, $G'=(V'=V\cup C, E')$ is a plane graph such that $C$ is the set of edge crossings in $D$,
and $E'$ consists of the crossing-free edges in $D$ and the edges which are obtained by
subdividing the other (crossing) edges using the crossing points. We put the charge $ch(v)=d_v-4$
on each vertex $v$ in $V'$, and the charge $ch(f)=d_f-4$ on each face in $F'$, where $F'$ is the set of faces of $G'$.
By Euler's formula the total sum of the charges is:
\begin{equation}
\label{eqn:mainSum}
\sum_{v\in V'}ch(v)+\sum_{f\in F'}ch(f)=-8.
\end{equation}
Indeed, $\sum_{v\in V'}(d_v-4)+\sum_{f\in F'}(d_f-4)= 2|E'|-4|V'|+2|E'|-4|F'|=-8$.

Moreover, as the charge at a vertex in $C$ is 0, we have:
\begin{equation}
\label{eqn:mainSum2}
\sum_{v\in V}ch(v)+\sum_{f\in F'}ch(f)=-8.
\end{equation}

In what follows we redistribute the charge in $G$ from some vertices in $V$ to some faces in $F'$, so that
all the faces have non-negative charge, and the charge left at any vertex is not "very" low.
 We maintain the total sum of the charge in the graph
unchanged.
 
Since every face in $F'$ of size at least $4$ receives a non-negative charge, 
it is enough to take care of the triangles (which initially have charge $-1$).
By the fact that all edges are represented by straight line segments, and every pair of them can cross
only at a right angle, each triangle $f$ must contain at least two vertices $u,v\in V$ on the boundary.
We discharge 1/2 of the charge at $u$ and $v$ to $f$, thereby making the charge of $f$ equal to 0.
It's easy to see that after doing this for every triangle the charge left at any vertex $v$ is still at least 
$d_v-4-\frac{1}{2}d_v=\frac{1}{2}d_v-4$. Let $ch'(v)$ and $ch'(f)$ denote the charge at each vertex $v\in V$ and 
$f\in F'$ after previously described redistribution. We have:
\begin{equation}
\label{eqn:proof}
m-4n=\sum_{v\in V}(\frac{1}{2}d_v-4)\leq\sum_{v\in V}ch'(v)\leq \sum_{v\in V}ch'(v)+\underbrace{\sum_{f\in F'}ch'(f)}_{\geq 0}=-8.
\end{equation}
By reordering the terms in (\ref{eqn:proof}) the result follows.
\end{proof}

\section{Proof of Theorem \ref{thm:R1}}

\begin{figure}[h]
\centering
\subfigure[]{\includegraphics[scale=0.55]{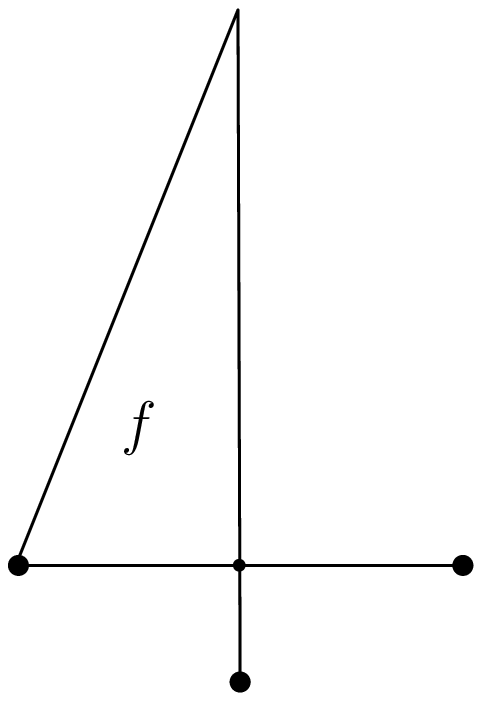}
	\label{fig:badlens3}  \hspace{5mm} 
	}
  \subfigure[]{\label{fig:badlens}     
		\includegraphics[scale=0.55]{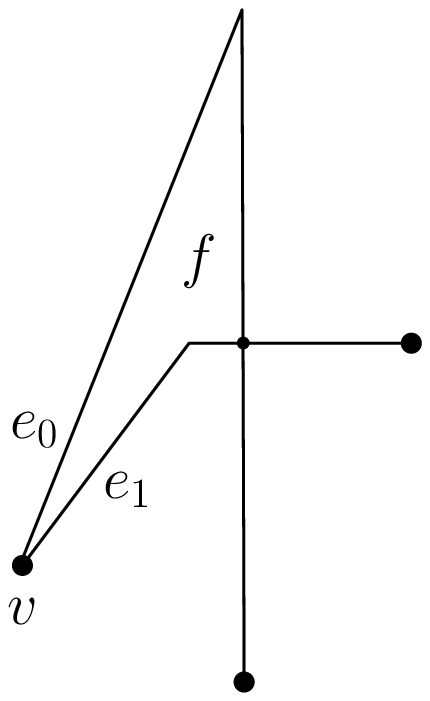}  \hspace{5mm} 
      }
      \subfigure[]{\includegraphics[scale=0.55]{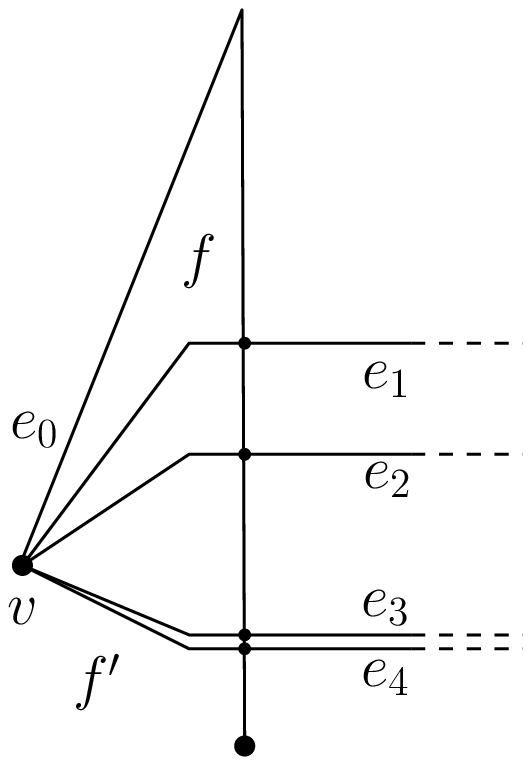}  \hspace{5mm} 
	\label{fig:badlens2}
	}
	\subfigure[]{\includegraphics[scale=0.55]{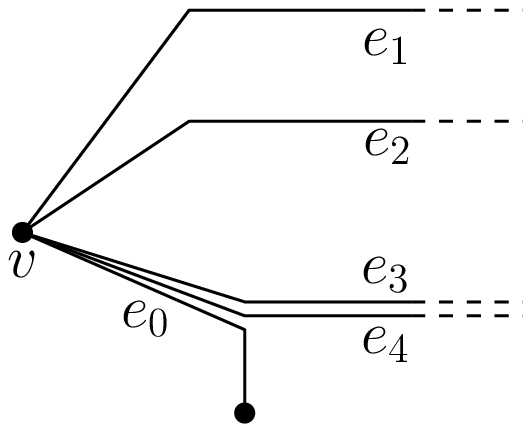}
	\label{fig:redrawing}
	}
	
	 \caption{(a) lens $f$ that can be redrawn,
(b) lens $f$ having only one convex bend on its boundary, (c) situation when $G$ could be redrawn with less crossings and (d) its redrawing ($i$=4)} 
\end{figure}

\begin{proof} [Proof of Theorem \ref{thm:R1}]
Fix a drawing $D$ of $G$ in the plane witnessing its membership to $R_1$, which minimizes the number of edge crossings.
Let $G_0=(V,E_0)$ and $D_0$ denote a graph and its drawing in the plane, respectively, which is obtained from $D$ by deleting all the crossing-free edges. Since the deleted edges form a plane drawing, there are at most $3n-6$ of them.
Similarly, as in the proof of Theorem \ref{thm:R0} we denote by $G'_0=(V_0', E_0')$ a plane multigraph, which is naturally obtained from $G_0$ by introducing
the vertices that were the crossings in $D_0$. Let $F_0'$ denote the set of faces of $G_0'$.

Following the line of thought of the proof of Theorem \ref{thm:R0} we put a charge $ch(v)=d_v-4$ and $ch(f)=d_f-4$ at each vertex $v$
 in $V_0'$ and face $f$ in $F_0'$,
 respectively. We have:
 \begin{equation}
\label{eqn:mainSum}
\sum_{v\in V}ch(v)+\sum_{f\in F_0'}ch(f)=-8.
\end{equation}

We redistribute the charge in $G_0'$ in two steps: first, from some vertices to some faces, and second, from some faces to other faces, so that in the end all the faces have non-negative charge. 

At the first step, for each pair $(v,e)$, where $v\in V$, $e\in E_0$ and $e$ is incident to $v$, we 
 discharge $1/2$ units of the charge from $v$ to the face $f\in F_0'$ that contains the
 convex bend of the edge $e$ (or the edge $e'\in E_0'$ that is a part of the edge $e$ in $G_0'$) on the boundary.
Observe that after this redistribution at each vertex $v\in V$ the new charge is $ch'(v)\geq\frac{1}{2}d_v-4$.

Since every face in $F'$ of size at least $4$ receives a non-negative charge already at the beginning, 
it is enough to take care of the triangles and lenses (the faces of size 2), which initially had charge $-1$ and $-2$, respectively.

We claim that after the first step of discharging all triangles have non-negative charges.
Indeed, by the fact that the sum of the inner angles in a simple closed polygon on $k$ vertices equals to $(k-2)\pi$,
and by the fact that every edge of $G_0$ participate in a crossing,
any triangle in $F_0'$ must contain a convex bend on its boundary. Thus, we  discharged to any triangle $f$ the charge
of 1/2, from both endpoints of any edge that creates a convex bend on $f$, thereby setting the new charge $ch'(f)$ at $f$ to a non-negative number.


Also, 
all lenses that contain two convex bends on its boundary obviously have non-negative charges (in fact, their new charge is 0).

We are left with the case of lenses
that contain at most one convex bend on their boundary. We claim that any such lens must look like the lens in  Figure \ref{fig:badlens}. 
Indeed, the other possible drawing of a lens with only one convex bend on the boundary (see Figure \ref{fig:badlens3})
could be easily redrawn so that we reduce the total number of crossings in the drawing $D$ (contradiction).

In order to set charge of the lenses with at most one convex bend we do the second step of discharging. Since in the first step we added charge to some faces of size at least 4 (which was unnecessary), we can now use that 'wasted' charge for the lenses. 

Let $f$ be a lens with at most one convex bend. Note that $f$ contains precisely one vertex $v$ from $V$ on its boundary.
Let $e_0,e_1,\ldots e_{d_v-1}$ denote the edges incident to $v$ listed according to the rotation system at $v$ (clockwise) so that
the wedge $(e_0,e_1)$ is contained in $f$ and $e_1$ creates the concave bend on $f$.
 Let $i\geq 1$ denote the minimum $i$ such that the wedge $(e_i,e_{i+1})$ is not contained
in a triangle (hereafter indices are taken modulo $d_v$). It is easy to see that $i$ is well-defined, as there exist no crossing-free 
edges in $G_0$. If $(e_{i-1},e_{i})$ is contained in a triangle that
has a convex bend created by $e_{i}$ on the boundary or does not have a bend created by $e_{i}$ on the boundary,
  we could redraw $G_0'$ and thereby reduce
 the number of crossings of $G$, a contradiction (see Figures \ref{fig:badlens2},\ref{fig:redrawing} for illustrations on how to redraw $G_0'$).

Thus, the wedge $(e_i,e_{i+1})$ is contained in a face $f'\not=f$ having a convex bend created by $e_i$ on the boundary. Moreover,
$d_{f'}>3$. Hence, we can use this bend to charge 1 to $f$ in order to make the charge $ch'(f)$ equal to 0.
It is easy to see that the situation, when the unused charge from the bend on $e_i$ is used for more than one lens, cannot happen. 

Finally, after the second series of redistributions, all the faces have non-negative charge.

Thus, by the same calculation as in (\ref{eqn:proof}), we get $|E_0|\leq 4n-8$, and that in turn implies that $$|E|\leq 4n-8+3n-6=7n-12.$$

We can complement the upper bound on the number of edges of a graph in $R_1$ by  constructing  infinitely many graphs 
belonging to $R_1$ and having $4.5n-O(\sqrt{n})$ edges. Our construction is a hexagonal lattice with 6 diagonals in each hexagon, see Figure \ref{fig:hexagonalium}.
The diagonals are obtained by erecting an isosceles right-angled triangle above each side of a hexagon and prolonging  the catheti.
Thus, by Theorem \ref{thm:R0} for infinitely many $n$ there is a graph on $n$ vertices belonging to $R_1$ and not to $R_0$.
\end{proof}
\begin{figure}
\centering
\includegraphics[scale=0.5]{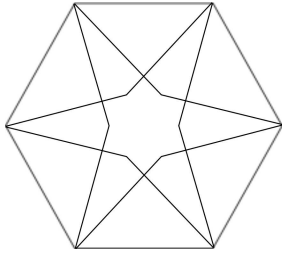}
\caption{Hexagon with diagonals} 
\label{fig:hexagonalium}
\end{figure}

\section{Proof of Theorem \ref{thm:R2}}

Contrary to the proof from the previous section we derive  the bound on the number of edges in a graph $G$ in  $R_2$ by
using a divide-and-conquer method based on Theorem \ref{thm:bw} \cite{Pach} below, which states that the bisection width of $G$
can be bounded from above in terms of the crossing number of $G$. 

For applying the divide-and-conquer strategy we bound from above the crossing number.

\begin{lemma}
\label{lemma:oneEdge}
If $G$ is a graph in $R_2$, then $cr(G)=O(nm+n^2)$, where $n$ is the number of vertices of $G$, and $m$ is the number of edges.
\end{lemma}

\begin{proof} 
Let us fix a drawing $D$ of $G=(V,E)$, which witnesses its membership to the class $R_2$.
For an $e\in E$, we distinguish two types of line segments, which $e$ consists of, the end segment and the middle segment.
Naturally, the end segment is a segment containing an endpoint of $e$, and the middle segment is the segment, which does not contain
an endpoint. Obviously, $e$ contains at most two end segments and one middle segment.

Without loss of generality we can assume that in $D$ no edge crosses itself.
We distinguish three types of crossings:

\begin{enumerate}[(i)]
\item
between two end segments,
\item
between an end segment and a middle segment,
\item
between two middle segments.
\end{enumerate}

First, we show that by deleting at most half of the edges we can destroy all the crossings of the type (iii).
We construct a bipartite graph $G_M$ whose set of vertices is the set of edges of $G$ and two edges of $G_M$ are joined by an edge,
if they give rise to a crossing of the type (iii). Clearly, by deleting from $G$ the edges corresponding to the smaller of two parts, into which the vertices
of $G_M$ are divided, we remove all the crossings of the type (iii) from $G$. Thus, we can assume that in $G$ there are no crossings of the type (iii).

There are at most $2{n \choose 2}$ crossings of the type (i), since by Thales' theorem for any pair $u,v\in V$ there are at most two such crossings
between end segments incident to $u$ and $v$. Finally, we can also easily put an upper bound $2mn$ on the number of crossings of the type (ii). Indeed, for a given
edge $e$ and a vertex $v$ there can be at most one crossing between the middle segment of $e$ and an end segment incident to $v$.

\end{proof}

The following theorem from \cite{Pach}, which can be found also  in \cite{Pach2}, allows us to apply a divide-and-conquer strategy to tackle our problem.
Its proof combines a weighted Lipton-Tarjan separator theorem for planar graphs \cite{Lipton} and results from \cite{Gazit}.

\begin{theorem}[\cite{Pach}]
\label{thm:bw}
Let $G$ be a graph on $n$ vertices with degrees $d_1,\dots,d_n$. Then
$$b(G)\leq1.58\left(16cr(G)+\sum_{i=1}^nd_i^2\right)^{1/2},$$
where $b(G)$ and $cr(G)$ denote the bisection width and the crossing number of $G$, respectively. 
\end{theorem}

In fact, we use Theorem \ref{thm:bw} in a similar way as it was used in \cite{Pach}.

Now, we are ready to prove Theorem \ref{thm:R2}.

\begin{proof} [Proof of Theorem \ref{thm:R2}]
We prove by induction on $n$ that $m\leq cn\log^2n$ holds for an appropriate constant $c>1$,
which will not be stated explicitly. 
For the base case, one can easily see that the theorem holds for the graphs with at most $3$ vertices.

For the inductive case, suppose that the theorem holds for all graphs having fewer than $n\geq 3$ vertices.

Observe that
$$\sum_{i=1}^nd_i^2\leq\sum_{i=1}^nd_in=2mn$$ holds for every graph. By Lemma \ref{lemma:oneEdge} we also have $cr(G)=O(nm)$ (we can suppose that $m\geq n$). Hence, 
by Theorem \ref{thm:bw}, it follows that $b(G)=O(\sqrt{nm})$. Let us assume that $b(G)\leq d\sqrt{nm}$ for some constant $d>0$.

Consider a partition of $V(G)$ into two parts $V_1$ and $V_2$, so that $|V_1|,|V_2|\leq2n/3$ and the number of edges between them is $b(G)$.
Let $G_1$ and $G_2$ denote the subgraphs of $G$ induced by $V_1$ and $V_2$. By the induction hypothesis both $G_1$ and $G_2$  belong to $R_2$.
Thus, we have the following:

\begin{equation}\label{eqn:bw} m\leq|E(G_1)|+|E(G_2)|+b(G)\leq cn_1\log^2n_1+cn_2\log^2n_2+d\sqrt{mn}\end{equation}
where $n_i=|V_i|$ ($i=1,2$). From (\ref{eqn:bw}) we get 
  
\begin{equation}\label{eqn:q} \sqrt{m}\leq\frac{d\sqrt n+\sqrt{d^2n+4cn_1\log^2n_1+4cn_2\log^2n_2}}{2}\end{equation}

Now it is enough to prove that we can choose $c$ large enough so that

\begin{equation}\label{eqn:l} \frac{d\sqrt n+\sqrt{d^2n+4cn_1\log^2n_1+4cn_2\log^2n_2}}{2}\leq\sqrt{cn\log^2n}\end{equation}
holds for every $n\geq 4$, since then (\ref{eqn:q}) and (\ref{eqn:l}) would imply that $m\leq cn\log^2n\,.$

Let us assume that $n_1=an$ and $n_2=bn$, where $a,b\in[1/3,2/3]$ and $a+b=1$\,.
By Jensen's inequality 
$$a\log^2(an)+b\log^2(bn)\leq\log^2[(a^2+b^2)n]\leq\log^2\left(\frac59n\right)\,,$$ 
so it is enough to find $c$ large enough so that 

\begin{equation}\label{eqn:ll} \frac{d+\sqrt{d^2+4c\log^2(5/9n)}}{2}\leq\sqrt{c\log^2n}\end{equation}
holds for every $n$. After some calculation the inequality (\ref{eqn:ll}) can be reduced to

\begin{equation}\label{eqn:lll} d\log n\leq\sqrt c\,[2\log(9/5)\log n-\log^2(5/9)]\end{equation}
It is easy to see that for large enough $c$ this inequality  holds for every $n\geq2$.
  \end{proof}

\section{Remarks}

The actual analysis of the drawing of $G$ in the proof of Theorem \ref{thm:R2} is undeniably very rough,
not to mention the application of the divide-and-conquer strategy.
Hence, one is prone to believe that the right order of magnitude is in this case linear as well.
This is also supported by our unsuccessful attempt to construct a super linear complementary lower bound.

Hence, to us it appears likely that a more clever and/or involved application of just
discharging method than that used in the proof of Theorem \ref{thm:R1} 
could yield a linear bound also in case of the class $R_2$.




\begin{thebibliography}{99}



\bibitem{Appel}
K. Appel, W. Haken, J. Koch: {\it Every Planar Map is Four Colorable}, Illinois Journal of Mathematics 21: 439–-567, 1977.

\bibitem{Appel2}
K. Appel, W. Haken: {\it Solution of the Four Color Map Problem}, Scientific American {\bf 237} (4): 108-–121, 1977.

\bibitem{Ackerman}
E. Ackerman, G. Tardos: {\it On the maximum number of edges in quasi-planar graphs},
J. Comb. Theory, Ser. A {\bf 114}(3): 563--571 (2007).

 
\bibitem{Ajtai}
M. Ajtai, V. Chv\'{a}tal, M. M. Newborn, and E. Szemer\'{e}di, {\it Crossing-free subgraphs},
 Theory and Practice of Combinatorics, volume 12 of Annals of Discrete Mathematics and volume 60 of North-Holland Mathematics Studies, 9--12. 1982. 

\bibitem{Dujmovic}
V. Dujmovi\'{c}, J. Gudmundsson, P. Morin, T. Wolle: {\it
Notes on large angle crossing graphs},
http://arxiv.org/abs/0908.3545


\bibitem{Didimo}
W. Didimo, P. Eades, G. Liotta: {\it Drawing Graphs with Right Angle Crossings}, WADS 2009: 206-217.

\bibitem{Gazit}
H. Gazit and G. L. Miller: {\it Planar separators and the Euclidean norm, Algorithms}, Proc. International
Symp. SIGAL '90 (T. Asano et al. eds.), Lecture Notes in Computer Science, Vol. 45(I),
Springer-Verlag, Berlin, 1990, pp. 338--347.


\bibitem{Hlineny}
P. Hlinen\'{y}: {\it Discharging technique in practice},
\url{http://kam.mff.cuni.cz/~kamserie/serie/clanky/2000/s475.ps}, (Lecture text for Spring School on Combinatorics),
(2000).
\bibitem{Lipton}
R. J. Lipton and R. E. Tarjan: {\it A separator theorem for planar graphs}, SlAM J. Appl. Math. {\bf 36}
(1979), 177--189.


\bibitem{Pach}
J. Pach, F. Shahrokhi, M. Szegedy: {\it Applications of the Crossing Number}, Algorithmica (1996) {\bf 16}:  1--117

\bibitem{Pach2}
J. Pach, P. Agarwal: {\it Combinatorial Geometry}, Wiley, 1995.

%


\bibitem{Wernicke}
P. Wernicke: {\it \"{U}ber den kartographischen Vierfarbensatz} (in German), Math. Ann. {\bf 58} (3): 413–-426,  1904.






\end{thebibliography}
\end{document}